%% file: main.tex
\documentclass[11pt]{amsart}
\input{preamble}

\usepackage{amsaddr}
\usepackage{mathrsfs}

\title[Avoiding Incr.\ Subseq.\ in Block-Ascending Permutations]
{Schur-Concavity for Avoidance of Increasing Subsequences in Block-Ascending Permutations}
\date{\today}
\author{Evan Chen}
\address{Department of Mathematics, Massachusetts Institute of Technology}
\email{evanchen@mit.edu}


\newcommand{\LL}{\mathcal L}
\newcommand{\DD}{\mathcal D}
\newcommand{\MV}{\mathbf V}
\newcommand{\MW}{\mathbf W}
\newcommand{\MWto}{\xrightarrow{\MW}}
\newcommand{\MVfrom}{\xleftarrow{\MV}}

\usepackage{ytableau}
\ytableausetup{mathmode,boxsize=3.141592653em}

\begin{document}

\begin{abstract}
	For integers $a_1, \dots, a_n \ge 0$ and $k \ge 1$,
	let $\mathcal L_{k+2}(a_1, \dots, a_n)$ denote the set of
	permutations of $\{1, \dots, a_1+\dots+a_n\}$
	whose descent set is contained in
	$\{a_1, a_1+a_2, \dots, a_1+\dots+a_{n-1}\}$,
	and which avoids the pattern $12\dots(k+2)$.
	We exhibit some bijections between such sets,
	most notably showing that $\# \mathcal L_{k+2} (a_1, \dots, a_n)$
	is symmetric in the $a_i$ and is in fact Schur-concave.
	This generalizes a set of equivalences observed by Mei and Wang.
\end{abstract}

\maketitle

\input{intro}
\input{easy}
\input{coreswap}
\input{lifting}
\input{enum}
\input{acknow}

\bibliographystyle{hplain}
\bibliography{refs}

\end{document}

%% file: preamble.tex
\usepackage{amsmath,amsthm,amssymb}
\usepackage{stmaryrd}
\usepackage[textsize=scriptsize,obeyFinal]{todonotes}
\usepackage{tikz-cd}

\newtheorem{theorem}{Theorem}[section]
\newtheorem*{theorem*}{Theorem}
\newtheorem{lemma}[theorem]{Lemma}
\newtheorem*{lemma*}{Lemma}
\newtheorem{proposition}[theorem]{Proposition}
\newtheorem*{proposition*}{Proposition}
\newtheorem{corollary}[theorem]{Corollary}
\newtheorem*{corollary*}{Corollary}

\theoremstyle{definition}
\newtheorem{definition}[theorem]{Definition}
\newtheorem*{definition*}{Definition}
\newtheorem{example}[theorem]{Example}
\newtheorem*{example*}{Example}

\newtheorem*{ques*}{Question}

\newtheorem*{claim*}{Claim}

\newtheorem*{remark*}{Remark}

\usepackage{bm}
\newcommand{\ii}{\item}

%% file: intro.tex
\section{Introduction}
\subsection{Synopsis}
For nonnegative integers $a_1$, \dots, $a_n$
an \emph{$(a_1, \dots, a_n)$-ascending permutation} is a permutation on
$\{1, 2, \dots, a_1 + \dots + a_n\}$ whose descent set is contained in
$\{a_1, a_1+a_2, \dots, a_1+\dots+a_{n-1}\}$.
In other words the permutation ascends in blocks of
length $a_1$, $a_2$, \dots, $a_n$,
and thus has the form
\[ \pi = \pi_{11} \dots \pi_{1a_1} \mid
	\pi_{21} \dots \pi_{2a_2} \mid \dots
	\mid \pi_{n1} \dots \pi_{na_n} \]
for which $\pi_{i1} < \pi_{i2} < \dots < \pi_{ia_i}$ for all $i$.
(The $\mid$ separators are added between blocks for readability.)
These permutations were introduced at least as early as 1993,
when Gessel and Reutenauer \cite{5Gessel} exhibited a bijection
between such permutations and so-called \emph{ornaments},
preserving the cycle structure of $\pi$.
Their work was then extended by others \cite{4EFW,6HanXin,steinhardt}.

In this paper we study such permutations,
but focusing on pattern avoidance rather than cycle structure.
\begin{definition}
	Let $a_1$, \dots, $a_n$ be nonnegative integers.
	\begin{itemize}
	\ii Let $\LL_{k+2}(a_1, \dots, a_n)$ denote the set of
	$(a_1, \dots, a_n)$-ascending permutations
	which avoid the pattern $12 \dots (k+2)$.
	In particular, $\LL_{k+2}(a_1 ,\dots, a_n) = \varnothing$
	if $\max \{a_1, \dots, a_n\} \ge k+2$.
	(The use of $k+2$ here is for consistency with
	\cite{lewis} and \cite{meiwang}.)
	\ii Let $\DD_h(a_1, \dots, a_n)$ denote the set of
	$(a_1, \dots, a_n)$-ascending permutations
	which avoid $12\dots(h+1)$ but not $12\dots h$,
	that is, the longest increasing subsequence
	should have length exactly equal to $h$.
	In other words, 
	\[ \DD_h(a_1, \dots, a_n)
		= \LL_{h+1}(a_1, \dots, a_n) \setminus \LL_h(a_1, \dots, a_n). \]
	\end{itemize}
\end{definition}

Many special cases of $\LL_{k+2}(a_1, \dots, a_n)$ are well-studied.
For example,
\begin{itemize}
	\ii $\LL_{3}(1, \dots, 1)$ is the
	set of $123$-avoiding permutations on $\{1, \dots, n\}$, and
	\ii $\LL_3(2, \dots, 2)$ is the set of alternating or
	``zig-zag'' permutations on $\{1, \dots, 2n\}$
	which avoid $123$.
\end{itemize}
Both have cardinality equal to the $n$th Catalan number.

In 2011, Lewis \cite[Proposition 3.1, Theorem 4.1]{lewis}
generalized these results to give two bijections:
\begin{itemize}
	\ii $\LL_{k+2} (k, \dots, k)$ 
	to standard Young tableaux of shape $\left< (k+1)^n \right>$, and
	\ii $\LL_{k+2} (k+1, \dots, k+1)$
	to standard Young tableaux of shape $\left< (k+1)^n \right>$.
\end{itemize}
His proof uses a modified version of the Robinson-Schensted-Knuth correspondence;
the hook-length formula then lets us compute the cardinalities.

In 2017, Mei and Wang \cite{meiwang}
generalized Lewis's bijections to the $2^n$ sets of the form
\begin{equation}
	\LL_{k+2}(a_1, \dots, a_n) \qquad\text{where}\qquad a_i \in \{k, k+1\}.
	\label{eq:yuanti}
\end{equation}
Thus this cardinality of such sets does not depend on the choice of
which $a_i$ are equal to $k$ or $k+1$ \cite[Theorem 2.3]{meiwang}.
Mei and Wang then proposed the problem of finding a direct bijection
between these sets of permutations, without appealing to the RSK correspondence
\cite[Problem 4.2]{meiwang}.\footnote{Actually,
	in the statement of Problem 4.2 in E-JC 24(1) 2017,
	there is a benign typo:
	$S_{nk}(123)$ should be replaced by just $\LL(n;k;\emptyset)$
	(which corresponds to $\LL_{k+2}(k,\dots,k)$ in our notation).
	In any case, our approach does not treat
	$\LL_{k+2}(k,\dots,k)$ specially.}

\subsection{Statement of Results}
The two major results we will prove are:
\begin{theorem}
	\label{thm:Dswap_main}
	For each $h$, the cardinality of $\DD_h(a_1, \dots, a_n)$
	does not depend on the order of the $a_i$'s,
	and there is an explicit bijection between the sets.
\end{theorem}
\begin{theorem}
	\label{thm:karamata}
	Fix $h$, and suppose the sequence $a_1 \ge \dots \ge a_n$
	majorizes the sequence $b_1 \ge \dots \ge b_n$.
	Then there is an explicit injection
	\[ \#\DD_h\left( a_1, a_2, \dots, a_n \right)
		\hookrightarrow \#\DD_h\left( b_1, b_2, \dots, b_n \right). \]
\end{theorem}
(Recall that a sequence $a_1 \ge \dots \ge a_n$
\emph{majorizes} a sequence $b_1 \ge \dots \ge b_n$
if $a_1 + \dots + a_i \ge b_1 + \dots + b_i$ for all $i$
and $a_1 + \dots + a_n = b_1 + \dots + b_n$.)
In other words, $\#\DD_h$ is Schur-concave.

\medskip

Because
$\LL_{k+2}(a_1, \dots, a_n) = \bigcup_{h \le k+1} \DD_h(a_1, \dots, a_n)$
the Schur-concavity holds for $\#\LL_{k+2}$ as well:
\begin{corollary}
	\label{cor:Lswap_main}
	For each $k$, the cardinality of $\LL_{k+2}(a_1, \dots, a_n)$
	does not depend on the order of the $a_i$'s,
	and there is an explicit bijection between the sets.
\end{corollary}
\begin{corollary}
	Fix $k$, and suppose the sequence $a_1 \ge \dots \ge a_n$
	majorizes the sequence $b_1 \ge \dots \ge b_n$.
	Then there is an explicit injection
	\[ \#\LL_{k+2}\left( a_1, a_2, \dots, a_n \right)
		\hookrightarrow \#\LL_{k+2}\left( b_1, b_2, \dots, b_n \right). \]
\end{corollary}

We will also make the following simple observation:
\begin{lemma}
	For all $k$, $a_2$, \dots, $a_n$,
	\[ \#\LL_{k+2}(k+1, a_2, \dots, a_n) = 
		\#\LL_{k+2}(k, a_2, \dots, a_n) \]
	and there is an explicit bijection between these sets.
	\label{lem:destriv}
\end{lemma}

The proofs of Theorem~\ref{thm:Dswap_main} (hence Corollary~\ref{cor:Lswap_main})
and Lemma~\ref{lem:destriv} are explicit bijections,
not relying on the RSK correspondence.
Hence these two results resolve Mei and Wang's problem
\cite[Problem 4.2]{meiwang},
because by composing them appropriately we may obtain a direct bijection
between any two sets of the form described in \eqref{eq:yuanti}.

\subsection{Outline}
The rest of the paper is structured as follows.
First in Section~\ref{sec:destriv} we quickly prove Lemma~\ref{lem:destriv}.
Then, in Section~\ref{sec:WV} we describe two maps $\MW$ and $\MV$
in the special situation $n = 2$, which will form the core of the proof.
In Section~\ref{sec:mainpf} we show how to extend the maps $\MW$ and $\MV$
in order to obtain the desired bijection.
Finally in Section~\ref{sec:notmuch} we
compute some specific values of $\#\LL_{k+2}(a_1, \dots, a_n)$.

%% file: easy.tex
\section{Proof of Lemma~\ref{lem:destriv}}
\label{sec:destriv}

First, we make the following observation.

\begin{lemma}
	If $\pi \in \LL_{k+2}(k+1, a_2, \dots, a_n)$ then
	$\pi_{1, k+1}$ is the largest element of $\pi$,
	that is, $\pi_{1,k+1} = (k+1) + a_2 + \dots + a_n$.
\end{lemma}
\begin{proof}
	By definition $\pi_{1,1} < \dots < \pi_{1,k+1}$.
	Moreover if $i \ge 2$ and $\pi_{i,j} > \pi_{1,k+1}$
	then $\pi_{1,1} < \dots < \pi_{1,k+1} < \pi_{i,j}$
	would be a $12\dots(k+2)$ pattern.
\end{proof}

This gives us the map
\begin{align*}
	\LL_{k+2}(k+1, a_2, \dots, a_n) &\to \LL_{k+2}(k, a_2, \dots, a_n) \\
	\intertext{defined by}
	\pi_{1,1} \dots \pi_{1,k} \pi_{1,k+1} \mid \pi_{2,1} \dots
	&\mapsto \pi_{1,1} \dots \pi_{1,k} \mid \pi_{2,1} \dots
\end{align*}
where we simply delete the maximal element from the $(k+1)$st position.
This map obviously admits an inverse,
since inserting a maximal element in the $(k+1)$st position
cannot introduce a $1\dots(k+2)$ pattern.
This produces the claimed bijection.

%% file: coreswap.tex
\section{The Bijections $\MW$ and $\MV$}
\label{sec:WV}

In this section we define two maps $\MW$ and $\MV$ between
sets of the form $\DD_h(p,q)$ for a fixed $h$.
These maps form the heart of the proof of Theorem~\ref{thm:Dswap_main}.

First, we introduce some notation for permutations of $\DD_h(p,q)$,
where $0 \le p,q \le h$. Consider a permutation
\[ \pi = x_1 x_2 \dots x_{p} \mid y_q y_{q-1} \dots y_1 \in \DD_h(p, q). \]
As the maximal increasing subsequence of $\pi$ has length $h$,
there should be an index $j$ such that
\begin{equation}
	x_1 < \dots < x_j < y_{h-j} < \dots < y_1.
	\label{eq:ridge}
\end{equation}
However, this $j$ may not be unique;
for example, $1368\mid2457$ has two maximal increasing subsequences,
namely $12457$ and $13457$.
Nonetheless, we are interested in the \emph{largest}
and \emph{smallest} indices with this property.
\begin{definition}
	For $\pi \in \DD_h(p,q)$, we denote by $\nu_h(\pi)$ and $\omega_h(\pi)$
	the smallest and largest index $j$, respectively,
	which satisfies \eqref{eq:ridge}.
\end{definition}

With this definition we may define the map $\MW$.
\begin{definition}
	Suppose $p \in \{0, \dots, h-1\}$ and $q \in \{1, \dots, h\}$.
	We define the map
	\[ \DD_h(p, q) \MWto \DD_h(p+1, q-1) \]
	by
	\begin{align*}
	\pi &= x_1 \dots x_p \mid y_q \dots y_1 \in \DD_h(p, q) \\
	\mapsto \MW(\pi) &= x_1 \dots x_j y_{h-j} x_{j+1} \dots x_p \mid
		y_q \dots y_{h-j+1} y_{h-j-1} \dots y_1
	\end{align*}
	where $j  = \omega_h(\pi)$.
	\label{def:W}
\end{definition}
In other words (in the notation of Definition~\ref{def:W}),
\begin{equation}
	x_1 < \dots < x_j < y_{h-j} < \dots < y_1
	\label{eq:maxsubseq}
\end{equation}
is an increasing subsequence of maximal length.
Observe that this requires $y_{h-j} = x_j + 1$
(or $y_{h-j} = 1$ if $j=0$),
since otherwise $y_{h-j}-1$ could be inserted into \eqref{eq:maxsubseq}.

\begin{example}
	For $(p,q)=(3,5)$, $h=6$, we have an example
	\begin{align*}
		\DD_6(3,5) &\MWto \DD_6(4,4) \\
		236 \mid 14578 &\mapsto 2346 \mid 1578.
	\end{align*}
\end{example}

\begin{proposition}
	This map is well-defined; that is, 
	the longest increasing subsequence of $\MW(\pi)$
	has length $h$.
\end{proposition}
\begin{proof}
	Assume not, and that moving $y_{h-j}$ introduces
	some increasing subsequence with length $h+1$.
	Then there must be some index $k$ such that
	\begin{align*}
		x_1 &< \dots < x_j < y_{h-j} < x_{j+1} < \dots < x_k \\
		&< y_{h-k} < y_{h-k-1} < \dots < y_1.
	\end{align*}
	But then $x_1 < \dots < x_k < y_{h-k} < \dots < y_1$
	is an increasing subsequence of length $h$ in $\pi$,
	contradicting the choice of $j = \omega_h(\pi)$ being maximal.
\end{proof}

The map $\MV$ is defined in an analogous way, in the reverse direction.
\begin{definition}
	Suppose $p \in \{1, \dots, h \}$ and $q \in \{0, \dots, h-1\}$.
	We define the map
	\[ \DD_h(p-1, q+1) \MVfrom \DD_h(p,q) \]
	by
	\begin{align*}
	\pi &= x_1 \dots x_p \mid y_q \dots y_1 \in \DD_h(p, q) \\
	\mapsto \MV(\pi) &= x_1 \dots x_{j-1} x_{j+1} \dots x_p \mid
		y_q \dots y_{h-j+1} x_j y_{h-j} y_{h-j-1} \dots y_1
	\end{align*}
	where $j  = \nu_h(\pi)$.
\end{definition}

In exactly the same way as before we have the following.
\begin{proposition}
	This map is well-defined; that is,
	the longest increasing subsequence of $\MV(\pi)$
	has length $h$.
\end{proposition}

\begin{proposition}
	The maps $\MW$ and $\MV$ are inverses, and hence bijections.
\end{proposition}
\begin{proof}
	We will check that $\MV(\MW(\pi)) = \pi$,
	with the other direction being analogous.
	Write 
	\begin{align*}
	\pi &= x_1 \dots x_p \mid y_q \dots y_1 \in \DD_h(p, q) \\
	\mapsto \MW(\pi) &= x_1 \dots x_j y_{h-j} x_{j+1} \dots x_p \mid
		y_q \dots y_{h-j+1} y_{h-j-1} \dots y_1
	\end{align*}
	where $j = \omega(\pi)$.
	Now, observe that $\MW(\pi)$ still has a subsequence
	\[ x_1 < \dots < x_j < y_{h-j} < y_{h-j+1} < \dots < y_1 \]
	and consequently, we have $\nu(\MW(\pi)) \le j + 1$.

	We now contend that $\nu(\MW(\pi)) = j+1$.
	(Informally, this is because all length $h$ subsequences of
	smaller index in the original sequence relied on $y_{h-j}$,
	and hence are killed by the application of $\MW$.)
	Assume for contradiction that $\nu(\MW(\pi)) < j+1$,
	so there is a $\ell \le j$ such that 
	\begin{align*}
		x_1 &< \dots < x_\ell \\
		&< y_{h-\ell+1} < \dots < y_{h-j+1} < y_{h-j-1} < \dots < y_1
	\end{align*}
	is an increasing subsequence in $\MW(\pi)$.
	But this would imply that
	\begin{align*}
		x_1 &< \dots < x_\ell \\
		&< y_{h-\ell+1} < \dots < y_{h-j+1} < y_{h-j} < y_{h-j-1} < \dots < y_1
	\end{align*}
	is an increasing subsequence of length $h+1$ in $\pi$,
	which is a contradiction.
\end{proof}

By composing the bijection $\MV$, we deduce the following corollaries.
\begin{corollary}
	\label{cor:Dsumtwo}
	Let $h \ge 1$ and $p,q,p',q' \in \{1, \dots, h\}$
	such that $p+q=p'+q'$.
	Then \[ \# \DD_h(p,q) = \# \DD_h(p', q'). \]
\end{corollary}
Observe that this already implies Theorem~\ref{thm:Dswap_main}
(and hence Corollary~\ref{cor:Lswap_main}) in the case $n = 2$;
that is, composition of $\MW$ induces a map
\begin{equation}
	\LL_{k+2}(p,q) \MWto \LL_{k+2}(q,p)
	\label{eq:Lswaptwo}
\end{equation}
whenever $p < q$.

%% file: lifting.tex
\section{Proofs of Theorem~\ref{thm:Dswap_main} and Theorem~\ref{thm:karamata}}
\label{sec:mainpf}

\subsection{Structure Preservation Lemma}
First, we will make the following useful observation about the map $\MW$.
\begin{lemma}
	Let $\DD_h(p,q) \MWto \DD_h(p+1,q-1)$,
	and $\pi \in \DD_h(p,q)$.
	For $1 \le a < b \le p+q$, the following are equivalent:
	\begin{itemize}
		\ii There is an increasing subsequence of length $r$ in $\pi$
		consisting of only elements in the interval $[a,b]$.
		\ii There is an increasing subsequence of length $r$ in $\MW(\pi)$
		consisting of only elements in the interval $[a,b]$.
	\end{itemize}
	\label{lem:h_preserve}
\end{lemma}

\begin{proof}
	We will check only the forward direction,
	the reverse direction being analogous using $\MV$ in place of $\MW$.
	As always, let $j = \omega(\pi)$ and write
	\begin{align*}
	\pi &= x_1 \dots x_p \mid y_q \dots y_1 \in \DD_h(p, q) \\
	\mapsto \MW(\pi) &= x_1 \dots x_j y_{h-j} x_{j+1} \dots x_p \mid
		y_q \dots y_{h-j+1} y_{h-j-1} \dots y_1
	\end{align*}
	Clearly it suffices to consider subsequences which involve $y_{h-j}$,
	since any other subsequence remains intact under $\MW$.
	
	We claim that $y_{h-j+\ell} < x_{j-\ell+1}$ for $1 \le \ell \le j$.
	Indeed, if this was not the case, then we could construct a sequence
	of length greater than $h$ in $\pi$ by taking
	\[ x_1 < \dots < x_{j-\ell+1} < y_{h-j+\ell} < \dots < y_{h-j} < \dots < y_1. \]
	Thus, given any subsequence, if there are any $y$ terms less than $y_{h-j}$
	then we may replace them with corresponding $x$ terms instead.
	Explicitly, if our subsequence of length $r$ in $\pi$ is
	\[ a \le x_{i_1} < \dots < x_{i_2} < y_{h-j+\ell}
		< \dots < y_{h-j} < \dots < y_{i_3} \le b \]
	then in $\MW(\pi)$ we have
	\[ a \le x_{i_1} < \dots < x_{i_2} < x_{j-(\ell-1)} < \dots
		< x_j < y_{h-j} < \dots < y_{i_3} \le b. \]
	This proves the lemma.
\end{proof}

\subsection{Proof of Theorem~\ref{thm:Dswap_main}}
We are now ready to prove the following result,
which implies Theorem~\ref{thm:Dswap_main} directly.
\begin{theorem}
	\label{thm:adjswap}
	For any index $\ell$, if $a_\ell \le a_{\ell+1}$
	then we have a bijection
	\begin{equation}
		\DD_h\left( a_1, \dots, a_\ell, a_{\ell+1}, \dots, a_n \right)
		\to \DD_h\left( a_1, \dots, a_{\ell+1}, a_\ell, \dots, a_n \right)
		\label{eq:Dtwobig}
	\end{equation}
	obtained by applying $\MW$ in \eqref{eq:Lswaptwo}
	on only the $\ell$th and $(\ell+1)$st blocks,
	viewed as a permutation on $\{1, \dots, a_\ell + a_{\ell+1}\}$.
	The inverse map is given by applying $\MV$ in the same way.

	In other words, we may swap two adjacent $a_i$'s.
\end{theorem}
\begin{example}
	For an example with $\DD_5(1,2,4,1) \to \DD_5(1,4,2,1)$
	we have
	\begin{align*}
		1\mid37\mid2458\mid6 &\mapsto 1\mid347\mid258\mid6 \\
		&\mapsto 1\mid3478\mid25\mid6.
	\end{align*}
\end{example}
\begin{proof}[Proof of Theorem~\ref{thm:adjswap}]
	Each permutation in
	$\DD_h(a_1, \dots, a_{\ell}, a_{\ell+1}, \dots, a_n)$
	naturally induces a permutation of $\DD_r(a_\ell, a_{\ell+1})$
	for some $r \ge a_{\ell+1}$,
	by looking at the relative ordering of the $a_\ell + a_{\ell+1}$
	elements in these two blocks.
	(To be exact, $r$ is the length of the longest increasing subsequence
	among $\pi_{\ell 1} \dots \pi_{\ell a_\ell} \mid
	\pi_{(\ell+1)1} \dots \pi_{(\ell+1) \pi_{\ell+1}}$.)
	In this way, we obtain a partition
	\[ \DD_h(a_1, \dots, a_{\ell}, a_{\ell+1}, \dots, a_n)
		= \bigcup_{r \ge a_{\ell+1}} X_r \]
	where $X_r$ is the set of permutations in $\DD_h(a_1, \dots, a_n)$
	whose longest increasing subsequence among the $\ell$th
	and $(\ell+1)$st block has length exactly $r$.

	Similarly, each permutation in
	$\DD_h(a_1, \dots, a_{\ell+1}, a_{\ell}, \dots, a_n)$
	naturally induces a permutation of $\DD_r(a_{\ell+1}, a_\ell)$
	for some $r \ge a_{\ell+1}$.
	So in exactly the same fashion we partition the left-hand side as
	\[ \DD_h(a_1, \dots, a_{\ell+1}, a_{\ell}, \dots, a_n)
		= \bigcup_{r \ge a_{\ell+1}} Y_r \]
	with $Y_r$ denoting those permutations in the right-hand side
	whose longest increasing subsequence among the $\ell$th
	and $(\ell+1)$st block has length exactly $r$.

	We claim that applying $\MW$ as 
	described in Theorem~\ref{thm:adjswap}
	yields a bijection $X_r \to Y_r$.
	This follows from Lemma~\ref{lem:h_preserve}:
	the lemma then ensures that at each application of $\MW$,
	no $1\dots(r+1)$ patterns are created,
	nor are any $1\dots r$ patterns destroyed.
	So the image of this map on $X_r$ really does lie in $Y_r$,
	as claimed.

	In the same way we may use $\MV$ to define a map in the reverse direction.
	Since $\MW$ and $\MV$ are inverses,
	we have produced a bijection $X_r \to Y_r$.
	Putting these together for all $r \ge a_{\ell+1}$ gives
	the desired result.
\end{proof}



\subsection{Proof of Theorem~\ref{thm:karamata}}
In analogy to before, we will prove the following result,
which implies Theorem~\ref{thm:karamata}.
\begin{theorem}
	\label{thm:adjconcave}
	For any index $\ell$, if $a_{\ell+1} \ge a_{\ell} + 2$
	then we have an injective map
	\begin{equation}
		\DD_h\left( a_1, \dots, a_\ell, a_{\ell+1}, \dots, a_n \right)
		\hookrightarrow
		\DD_h\left( a_1, \dots, a_\ell+1, a_{\ell+1}-1, \dots, a_n \right)
		\label{eq:concaveswap}
	\end{equation}
	obtained by applying $\MW$ in \eqref{eq:Lswaptwo}
	on only the $\ell$th and $(\ell+1)$st blocks,
	viewed as a permutation on $\{1, \dots, a_\ell + a_{\ell+1}\}$.
\end{theorem}
\begin{proof}
	This is really an observation made within the proof of
	Theorem~\ref{thm:adjswap}.
	Retaining the notation in our earlier proof, we decompose
	\begin{align*}
		\DD_h(a_1, \dots, a_{\ell}, a_{\ell+1}, \dots, a_n)
		&= \bigcup_{r \ge a_{\ell+1}} X_r \\
		\DD_h(a_1, \dots, a_{\ell}+1, a_{\ell+1}-1, \dots, a_n)
		&= \bigcup_{r \ge a_{\ell+1}-1} Y_r.
	\end{align*}
	As in the proof of Theorem~\ref{thm:Dswap_main},
	we obtain bijections $X_r \to Y_r$ for $r \ge a_{\ell+0}$
	which collate to give a bijection
	\[ \bigcup_{r \ge a_{\ell+1}} X_r
		\to \bigcup_{r \ge a_{\ell+1}} Y_r. \]
	The change from the previous proof is that we now have a set
	$Y_{a_{\ell+1}-1}$ on the right-hand side
	which is not in the image of our map.
	Nonetheless we may still conclude our map is injective,
	which proves Theorem~\ref{thm:adjconcave}.
\end{proof}

%% file: enum.tex
\section{Enumeration}
\label{sec:notmuch}
Now that we have a symmetry result,
we turn our attention to actually computing
$\#\LL_{k+2}(a_1, \dots, a_n)$ in certain situations.
By the main result of this paper, it suffices to assume
\[ 1 \le a_1 \le \cdots \le a_n \le k. \]

The general problem of computing the value seems difficult,
since the special case $a_1 = \dots = a_n = 1$
is equivalent to computing the number of $12\dots(k+2)$
avoiding permutations; no closed formula is known for $k \ge 3$.
Nonetheless, even computing the cardinality for special cases
other than those for which $a_i \in \{k, k+1\}$ would be interesting.
We give some examples here.

\subsection{The $n=2$ Case}
We show that $\#\DD_h(p,q)$ is given by the entries of Catalan's triangle.
\begin{proposition}
	As usual, let
	\[ C(n,k) = \frac{(n+k)! (n-k+1)}{k!(n+1)!}
		= \binom{n+k}{k} - \binom{n+k}{k-1} \]
	denote the $(n,k)$th entry of Catalan's triangle.
	Then for any $1 \le p \le q \le h$, we have
	\[ \# \DD_h \left( p, q \right) 
		= \begin{cases}
			C(h, p+q-h) & p+q \ge h \\
			0 & p+q < h.
		\end{cases} \]
\end{proposition}
\begin{proof}
	Assume $p+q \ge h$, and let $m = p+q-h \ge 0$ for brevity.
	Thus by Corollary~\ref{cor:Dsumtwo}, we have
	\[ \# \DD_h(p,q) = \# \DD_h(h, m). \]
	If $m=0$ the result is clear so assume $m > 0$.
	We now prove the result by induction on $h+m$.
	From Lemma~\ref{lem:destriv} and the definition of $\LL_{h+1}$,
	\begin{align*}
		\# \LL_{h+1} (h, m) &= \# \LL_{h+1} (h-1, m). \\
		\# \DD_h (h,m) &= \# \DD_h (h-1,m) + \# \DD_{h-1} (h-1, m) \\
		&= C(h,m-1) + C(h-1, m) = C(h,m),
	\end{align*}
	which completes the inductive step.
	(The term $\DD_{h-1}(h-1,m)$ is omitted when $m=h$.)
\end{proof}

\subsection{Young Tableaux}
We now give some examples of possible generalizations of
the cardinality computed in \cite{meiwang}.
\begin{proposition}
	For $p \le k$, the cardinality
	\[ \# \LL_{k+2}(p, \underbrace{k, k, \dots, k}_{n-1}) \]
	is equal to the number of standard Young tableaux of shape
	\[ \left< (k+1)^{n-1}, p \right>. \]
\end{proposition}
Of course, this cardinality may be computed using the hook-length formula.
\begin{proof}
	This is essentially identical to
	\cite[Proposition 3.1]{lewis}.
	By our results, it suffices to consider the cardinality of
	\[ \LL_{k+2}(p, \underbrace{k+1, k+1, \dots, k+1}_{n-1}). \]
	Given a permutation
	$\pi = \pi_{11} \pi_{12} \dots \pi_{1p} \mid \dots \mid \pi_{n1} \dots \pi_{n(k+1)}$,
	we construct a tableau as follows:
	\begin{center}
	\begin{ytableau}
		\pi_{n,1} & \pi_{n,2} & \pi_{n,3} & \none[\dots] & \pi_{n,p} & \none[\dots] & \pi_{n,k+1} \\
		\pi_{n-1,1} & \pi_{n-1,2} & \pi_{n-1,3} & \none[\dots] & \pi_{n-1,p} & \none[\dots] & \scriptstyle \pi_{n-1,k+1} \\
		\none[\vdots] & \none[\vdots] & \none[\vdots] & \none[\ddots] & \none[\vdots] & \none[\ddots] & \none[\vdots] \\
		\pi_{2,1} & \pi_{2,2} & \pi_{2,3} & \none[\dots] & \pi_{2,p} & \none[\dots] & \pi_{2,k+1} \\
		\pi_{1,1} & \pi_{1,2} & \pi_{1,3} & \none[\dots] & \pi_{1,p}
	\end{ytableau}
	\end{center}
	Obviously each row is increasing;
	then, one observes that $\pi$ has no $12\dots(k+2)$ pattern exactly if
	the tableau $T$ is a standard Young tableau
	(the columns are increasing as well).
\end{proof}

\subsection{Skew Young Tableau}
It is possible to generalize both the results above
using the concept of skew Young tableaux.
\begin{proposition}
	For $p \le q \le k$, the cardinality
	\[ \# \LL_{k+2}(p, q, \underbrace{k, k, \dots, k}_{n-2}) \]
	is equal to the number of standard skew Young tableaux of shape
	\[ \left< (k+1)^{n-1}, p \right> / \left< k+1-q \right>. \]
\end{proposition}

\begin{proof}
	By our results, it suffices to consider the cardinality of
	\[ \LL_{k+2}(p, \underbrace{k+1, k+1, \dots, k+1}_{n-2}, q). \]
	Given a permutation
	$\pi = \pi_{11} \pi_{12} \dots \pi_{1p} \mid \dots \mid \pi_{n1} \dots \pi_{nq}$,
	we write it in an array as follows:
	\begin{center}
	\begin{ytableau}
		\none & \none & \pi_{n,1} & \none[\dots] & \scriptstyle \pi_{\scriptscriptstyle n,k+1-q} & \none[\dots] & \pi_{n,q} \\
		\pi_{n-1,1} & \none[\dots] & \resizebox{2.8em}{!}{$\pi_{n-1,k+2-q}$} & \none[\dots] & \pi_{n-1,p} & \none[\dots] & \scriptstyle \pi_{n-1,k+1} \\
		\none[\vdots] & \none[\vdots] & \none[\vdots] & \none[\ddots] & \none[\vdots] & \none[\ddots] & \none[\vdots] \\
		\pi_{2,1} & \none[\dots] & \scriptstyle \pi_{2,k+1-q} & \none[\dots] & \pi_{2,p} & \none[\dots] & \pi_{2,k+1} \\
		\pi_{1,1} & \none[\dots] & \scriptstyle \pi_{1,k+2-q} & \none[\dots] & \pi_{1,p}
	\end{ytableau}
	\end{center}
	In the same way as before,
	one observes that $\pi$ has no $12\dots(k+2)$ pattern exactly if
	this tableau has increasing columns.
\end{proof}
\begin{example}
	To compute $\# \LL_8(4,5,6,6,6)$, we biject it to $\LL_{8}(4,7,7,7,5)$
	and arrange the permutations of the latter in the following fashion:
	\begin{center}
	\ytableausetup{mathmode,boxsize=2.71828459em}
	\begin{ytableau}
		\none & \none & \pi_{51} & \pi_{52} & \pi_{53} & \pi_{54} & \pi_{55} \\
		\pi_{41} & \pi_{42} & \pi_{43} & \pi_{44} & \pi_{45} & \pi_{46} & \pi_{47} \\
		\pi_{31} & \pi_{32} & \pi_{33} & \pi_{34} & \pi_{35} & \pi_{36} & \pi_{37} \\
		\pi_{21} & \pi_{22} & \pi_{23} & \pi_{24} & \pi_{25} & \pi_{26} & \pi_{27} \\
		\pi_{11} & \pi_{12} & \pi_{13} & \pi_{14}
	\end{ytableau}
	\end{center}
	Thus, $\# \LL_8(4,5,6,6,6)$ is equal to the number of
	standard Young tableaux of shape
	$\left< 7,7,7,7,4\right> / \left< 2\right>$.
\end{example}

%% file: acknow.tex
\section*{Acknowledgments}
This research was funded by NSF grant 1659047,
as part of the 2017 Duluth Research Experience for Undergraduates (REU).
The author thanks Joe Gallian for supervising the research,
and for suggesting the problem.

The author also wishes to acknowledge
Joe Gallian, Mitchell Lee, Ben Gunby,
and the anonymous referee for their comments on drafts of the paper.
Special thanks to Ben Gunby for also pointing out the concavity result
(Theorem~\ref{thm:karamata}).